\documentclass[reqno]{amsart}
\usepackage{amssymb,amsmath,amsfonts,amsthm,graphicx,ifthen,mathrsfs}
\usepackage[mathcal]{euscript}
\usepackage{graphicx} 
\usepackage[dvipsnames]{xcolor} 
\usepackage{tikz}
\usepackage{float}
\usetikzlibrary{positioning}
\usepackage{xspace}  
\usepackage{textcase} 

\usepackage[hidelinks]{hyperref}

\theoremstyle{plain}
\newtheorem{Theorem}{Theorem} [section]
\newtheorem{Lemma}[Theorem] {Lemma}
\newtheorem{Corollary}[Theorem] {Corollary}

\theoremstyle{definition}
\newtheorem{Definition}[Theorem] {Definition}
\newtheorem{Construction}[Theorem] {Construction}
\newtheorem{Problem}[Theorem] {Problem}

\usepackage[capitalise, noabbrev]{cleveref}
\crefname{Construction}{Construction}{Constructions}
\Crefname{Construction}{Construction}{Constructions}
\crefname{Problem}{Problem}{Problems}
\Crefname{Problem}{Problem}{Problems}

\newcommand{\pak}{\ruleset{partizan arc kayles}}
\newcommand{\dom}{\ruleset{domineering}}
\newcommand{\snort}{\ruleset{snort}}

\newcommand{\strat}[1]{\ensuremath{\mathcal{#1}}\xspace}

\newcommand{\ruleset}[1]{\textup{\textsc{#1}}\xspace}

\usepackage{fullpage}

\newcommand{\tile}[2]{\fill[fill=gray] (#1,#2) rectangle (#1 + 1,#2 + 1);}

\begin{document}

\title{PA\NoCaseChange{rc}K trees can get as hot as you want}
\author{S. Huntemann, N. McKay, C. Tennenhouse}
\date{October 16, 2026}

\begin{abstract}

We consider the combinatorial game \pak as a generalization of \dom. An extensive technical search using genetic algorithms found hot positions of temperature exceeding $2$. We discuss this search and construct a family of positions with high temperature that includes these graphs and prove that \pak has unbounded temperature. We identify an infinite family of second-player win positions as part of our main construction.

\end{abstract}

\maketitle


\section{Introduction}\label{sec:Intro}

The combinatorial game \dom is played on a rectangular grid on which dominoes are placed alternately by the players, called Left and Right. Left places dominoes vertically; Right places dominoes horizontally. Upon being unable to play, that player loses. 

In a game of \dom, the board naturally breaks up into smaller components which are completely disjoint from each other. On their turn, a player then has to decide on which component to move and make their move there. 
	
A common question in combinatorial game theory is which of these components the next player should move in. Two tools used in this approach are the \textbf{game value} and the \textbf{temperature} of a component. The game value intuitively is how much of an advantage the winning player has in that component, while temperature indicates how urgent it is to make the first move in the component.

Most of the time, the two players will choose to move in a component with the highest temperature. Thus, both human and computer players look for components with high temperature, whether explicitly computing temperature (for the latter) or intuitively (for the former). As such, it is an advantage for a player if they know what  the highest possible temperature is in a game can be. The highest possible temperature in a game is also known as the \textbf{boiling point}. A search for a good move may be terminated when a component with temperature equal or close to the boiling point is found (a ``good enough'' position).

We introduce some background necessary for this paper within.
The reader, obviously interested in combinatorial game theory, can see also \cite{BerlekampCG2004,siegel2013combinatorial}.

A long-standing question first posed by Berlekamp is what temperatures occur in \dom (see for example \cite{guy1995unsolved,berlekamp2019}). Until recently it was conjectured that the boiling point is $2$. Drummond-Cole first published a position that indeed has temperature $2$ \cite{drummond2004temperature}, which is seen on the left in \cref{fig:dc}. Several other searches have found more positions with temperature $2$ \cite{shankar2005new,HuntemannM2026}. A search of positions assisted by AI very recently found a position of temperature $33/16=2+\frac{1}{16}$ \cite{mazur-2026-berlekamp-domineering-counterexample}, breaking the original conjecture, but still leaving the original question of which temperatures are possible open. The position is shown in \cref{fig:AI_high_temp}, and its temperature was independently verified by the authors of this paper. Note that this position contains the Drummond-Cole position, as is the case with every position of temperature $2$ previously found as well.

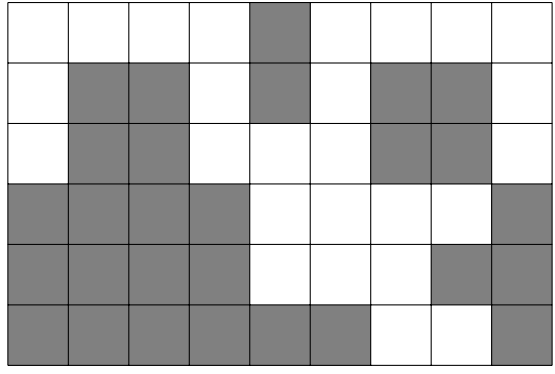
\begin{figure}[!h]
	\centering
	\begin{tikzpicture}[scale=0.8,every node/.style={circle,draw, draw=black,thin}]
		\tile{-3}{0}
		\tile{-3}{1}
		\tile{-3}{2}
		\tile{-2}{0}
		\tile{-2}{1}
		\tile{-2}{2}
		\tile{-2}{3}
		\tile{-2}{4}
		\tile{-1}{0}
		\tile{-1}{1}
		\tile{-1}{2}
		\tile{-1}{3}
		\tile{-1}{4}
		\tile{0}{0}
		\tile{0}{1}
		\tile{0}{2}
		\tile{1}{0}
		\tile{1}{4}
		\tile{1}{5}
		\tile{2}{0}
		\tile{3}{3}
		\tile{3}{4}
		\tile{4}{1}
		\tile{4}{3}
		\tile{4}{4}
		\tile{5}{0}
		\tile{5}{1}
		\tile{5}{2}
		
		\draw[step=1cm,black] (-3,0) grid (6, 6);
	\end{tikzpicture}
	\caption{A \textsc{domineering} position of temperature $33/16$}
	\label{fig:AI_high_temp}    
\end{figure}

At the Berlekamp Memorial Workshop on Combinatorial Games in 2019 at MSRI, a group of game theorists, including two of the authors, considered the question of whether the grid structure of a \dom board is at the root of \dom having a (potential) boiling point of $2$, with the question now updated to whether it has a finite boiling point.

To answer this question, we consider a generalization of \dom called \pak which can be played on any edge-colored graph. This game is also the partizan generalization of the impartial game \ruleset{ArcKayles} in which players take turns removing an edge and all incident edges. In another recent paper, the game was called \ruleset{Colored ArcKayles} \cite{Hanaka2026}. We prefer to use the adjective \emph{partizan} as is often done in combinatorial game theory.

\begin{Definition}\label{def:PAK}
	The ruleset \pak is played on a edge-colored graph. On their turn, Left deletes a blue edge and the two incident vertices, while Right deletes a red edge and the two incident vertices.
\end{Definition}

Every \dom game is also a game of \pak: given a \dom board, construct an edge-colored graph by replacing each square with a vertex, and add a blue (resp.\ red) edge between two vertices if they represent vertically (resp.\ horizontally) adjacent squares. One example of this equivalence is given in \cref{fig:dc}. As such, \pak is indeed a generalization of \dom.

\begin{figure}[!h]
	\centering
	\begin{tikzpicture}[scale=0.8,every node/.style={circle,draw, draw=black,thin}]
		\tile{0}{0}
		\tile{0}{1}
		\tile{0}{2}
		\tile{1}{0}
		\tile{1}{4}
		\tile{2}{0}
		\tile{3}{3}
		\tile{3}{4}
		\tile{4}{1}
		\tile{4}{3}
		\tile{4}{4}
		\draw[step=1cm,black] (0,0) grid (5, 5);
		
		\begin{scope}[shift={(8,0)}]
			\node (0a) at (0.5,3.5) {};
			\node (0b) at (0.5,4.5) {};
			\node (1a) at (1.5,1.5) {};
			\node (1b) at (1.5,2.5) {};
			\node (1c) at (1.5,3.5) {};
			\node (2a) at (2.5,1.5) {};
			\node (2b) at (2.5,2.5) {};
			\node (2c) at (2.5,3.5) {};
			\node (2d) at (2.5,4.5) {};
			\node (3a) at (3.5,0.5) {};
			\node (3b) at (3.5,1.5) {};
			\node (3c) at (3.5,2.5) {};
			\node (4a) at (4.5,0.5) {};
			\node (4b) at (4.5,2.5) {};
			\draw[blue,thick] (0a)--(0b) (1a)--(1b)--(1c) (2a)--(2b)--(2c)--(2d) (3a)--(3b)--(3c);
			\draw[red,thick] (0a)--(1c)--(2c) (1b)--(2b)--(3c)--(4b) (1a)--(2a)--(3b) (3a)--(4a);
		\end{scope}
	\end{tikzpicture}
	\caption{A temperature $2$ \textsc{domineering} position and its associated position in \pak}
	\label{fig:dc}    
\end{figure}
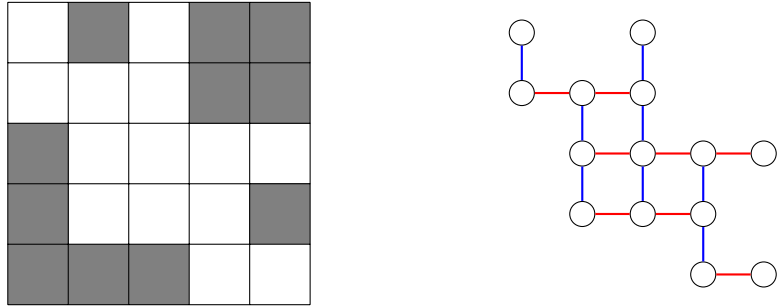

In \pak, in contrast to \dom, we do not have a restriction to a grid structure or even a bound on the maximum degree of a position. As such, one approach to answering the question about the grid structure of \dom being relevant to its boiling point is to determine the boiling point of \pak. We summarize our two main questions below.

\begin{Problem}\leavevmode
	\begin{enumerate}
		\item[(a)] What temperatures are possible in \dom ?
		\item[(b)] Does \pak have a finite boiling point?\label{prob:pak_boiling}
	\end{enumerate}	
\end{Problem}

\emph{Placement games} are games in which pieces are placed on the board without moving or removing them later in the game. We can think of \pak as a placement game if we consider removing an edge and its two endpoints instead as placing a piece on the edge and the two endpoints. \pak has the further property of being a \emph{strong} placement game \cite{FaridHN2019}, meaning that in any position the preceding moves could have been played in any order while still reaching the same value. We will state this fact as a lemma as we will be using it throughout our proofs later on.

\begin{Lemma}\label{thm:PAKstrong}
\pak is a strong placement game.
\end{Lemma}

\medskip

In Section \ref{sec:Search}, we discuss our computer search. In Section \ref{sec:Construction}, we present a construction of infinite families of positions of values $0$ and $1$.
We resolve Problem~\ref{prob:pak_boiling}(b) in \cref{sec:High} with a construction of a family of \pak positions of unbounded temperature.
We end in \ref{sec:Conclusion} with a discussion of the relevance of this work and directions for further study.

\section{Computer search}\label{sec:Search}

Game temperature is a dyadic rational-valued function on the space of game positions, so it lends itself to optimization methods. In order to search for high temperature positions in \pak we utilized a purpose-built genetic algorithm written in CG Script running within CGSuite \cite{CGSuite}. While at the time CGSuite did not have a pseudorandom number generator, we wrote one that generated dyadic rationals between $0$ and $1$ and chose a seed on each run. From this pseudorandom number generator we were able to generate random integers and a host of other stochastic tools. Due to the lack of support for certain data structures, we modelled a \pak position with an upper-triangular grid as a chromosome for the genetic algorithm, representing the adjacency matrix of an edge-colored tree. Our crossover function consisted of swapping equal-sized rectangles from the corners of a pair of chromosomes, and mutation was a simple change in the position or color of a single ``stone''. While these functions are not directly modelled after phenotypic changes to \pak positions, they nonetheless represented genotypic changes sufficient for evolution for the sake of the objective function, game temperature. 

Through many hundreds of generations of evolution we saw a common structure emerge in the form of a node with two blue-red (or two red-blue) paths (see $A_2$ in \cref{fig:base_cases}). This structure also appears in the well-known \dom position of temperature $2$ (\cref{fig:dc}) at the top and right side. After running the simulation longer, a position of temperature $\frac{5}{2}$ appeared. From this position, we were able to generalize to the temperature $3$ position $T_3$ in \cref{fig:T3}. Seeing that this position could be recursively extended, we generated and tested positions of temperature $4$ and $5$, the construction of which we explain in \cref{sec:Construction}.

\section{Infinite families of numbers in \pak}\label{sec:A}

\begin{figure}[!ht]
	\scalebox{0.75}{
		\begin{tikzpicture}
			[main node/.style={circle,draw=black,fill=white,text=black},
			main edge from parent/.style={draw,thick,draw=blue},every label/.style={text=black}]
			\node[main node, label=left:{$p''$}] (root) {}
			child[main edge from parent] {node[main node, yshift=-0.5cm, label=left:{$p'$}] (mid) {}
				child[main edge from parent,red] {node[main node, yshift=-0.5cm, label=left:{$p$}] (rootA) {} 
					child[main edge from parent] {node[main node, yshift=-1cm] (a1) {$-A_1$}}
					child[main edge from parent] {node[main node, yshift=-1cm] (a2) {$-A_1$}}
					child[draw=white] {node[yshift=-1cm] (x) {\ldots}}
					child[main edge from parent] {node[main node, yshift=-1cm] (a3) {$-A_{n-1}$}}
					child[main edge from parent] {node[main node, yshift=-1cm] (a4) {$-A_{n-1}$}}
			}};
		\end{tikzpicture}
	}
	\caption{The construction of $C_n$}
	\label{fig:Cn}
\end{figure}
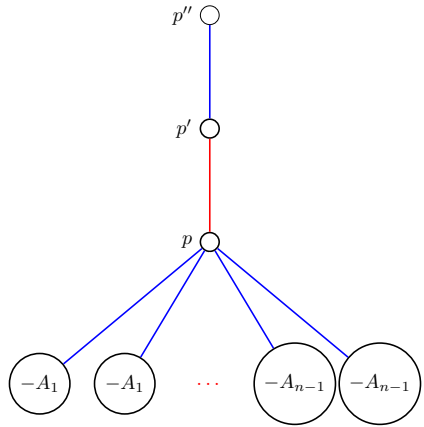

Consider the position shown in \cref{fig:Cn}. We ask the reader to consider one move for each player. Left's move in the top blue edge is to the subtree rooted at $p$, which we call $A_n$. 
Right can move in the red edge shown and he moves, for large $n$, to a large disjunctive sum of $-A_i$ of smaller indices. If $A_n \ge 0$ and enough of the $A_i$ are numbers and positive, then $C_n$ is hot. We will construct such a family of positions, $A_n$.

\subsection{Recursive Construction of the $A$ family of positions}\label{sec:Construction}

\begin{Construction}\label{def:An}
	Let $A_n$ be the \pak position composed of a root node $r$ and $2(n-1)$ blue edges from $p$ to the root vertex of each of $2$ copies of $-A_i$, where $-A_i$ is $A_i$ with the edge colors swapped, for  $i\in [1,n-1]$. Let $A_1$ be a single blue edge. Note that the depth of $A_n$ is $n$.
\end{Construction}

See \cref{fig:base_cases,fig:An} for the base cases and recursive construction from Definition \ref{def:An}. In an abuse of notation, we use $A_n$ in figures to represent the underlying graph that corresponds to the \pak position $A_n$. Later, $A_n$ may be read as being the value of the position. This ought to be clear from the context.
\begin{figure}[!ht]
	\scalebox{0.75}{
		\begin{tikzpicture}
			[level distance=10mm,
			every node/.style={circle,draw, draw=black,thin},
			edge from parent/.style={draw,thick},
			level 1/.style={sibling distance=40mm,edge from parent/.append style={blue}},
			level 2/.style={sibling distance=20mm,edge from parent/.append style={red}},
			level 3/.style={sibling distance=10mm,edge from parent/.append style={blue}}]
			\node {}
			child[xshift=0cm] {node (a) {}
			}
			;
		\end{tikzpicture}
		\qquad
		\begin{tikzpicture}
			[level distance=10mm,
			every node/.style={circle,draw, draw=black,thin},
			edge from parent/.style={draw,thick},
			level 1/.style={sibling distance=40mm,edge from parent/.append style={blue}},
			level 2/.style={sibling distance=20mm,edge from parent/.append style={red}},
			level 3/.style={sibling distance=10mm,edge from parent/.append style={blue}}]
			\node {}
			child[xshift=1cm] {node (a) {}
				child {node (e) {}}
			}
			child[xshift=-1cm] {node (b) {}
				child {node (f) {}}
			}
			;
		\end{tikzpicture}
		\qquad
		\begin{tikzpicture}
			[level distance=10mm,
			every node/.style={circle,draw, draw=black, thin},
			edge from parent/.style={draw,thick},
			level 1/.style={sibling distance=40mm,edge from parent/.append style={blue}},
			level 2/.style={sibling distance=20mm,edge from parent/.append style={red}},
			level 3/.style={sibling distance=10mm,edge from parent/.append style={blue}}]
			\node  {}
			child[xshift=4cm] {node  (a) {}
				child {node  (e) {}}
			}
			child[xshift=1cm] {node  (b) {}
				child {node  (f) {}}
			}
			child[xshift=-1.5cm]  {node  (c) {}
				child {node  (g) {}
					child {node  (k) {}}
				}
				child {node  (h) {}
					child {node  (l) {}}
				}
			}
			child[xshift=-2.5cm] {node  (d) {}
				child {node  (i) {}
					child {node  (m) {}}
				}
				child {node  (j) {}
					child {node  (n) {}}
				}
			};
		\end{tikzpicture}
	}
	\caption{The \pak\ game $A_1 + A_2 + A_3$}
	\label{fig:base_cases}
\end{figure}
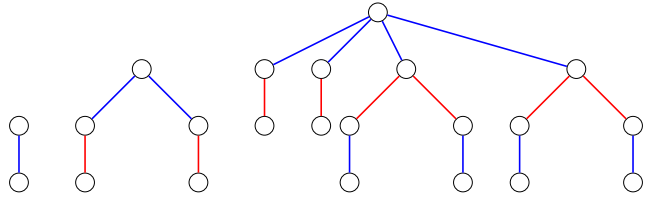

\begin{figure}[!ht]
	\scalebox{0.75}{
		\begin{tikzpicture}
			[main node/.style={circle,draw=black,fill=white,text=black},
			main edge from parent/.style={draw,thick,draw=blue}]
			\node[main node, label=left:{$r$}] {}
			child[main edge from parent] {node[main node, yshift=-1cm] (a1) {$-A_1$}}
			child[main edge from parent] {node[main node, yshift=-1cm] (a2) {$-A_1$}}
			child[draw=white] {node[yshift=-1cm] (x) {\ldots}}
			child[main edge from parent] {node[main node, yshift=-1cm] (a3) {$-A_{n-1}$}}
			child[main edge from parent] {node[main node, yshift=-1cm] (a4) {$-A_{n-1}$}};
		\end{tikzpicture}
	}
	\caption{The recursive construction of $A_n$}
	\label{fig:An}
\end{figure}
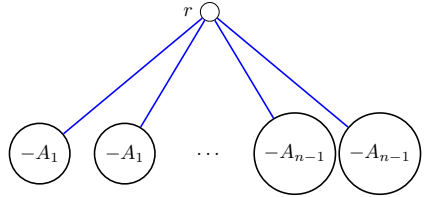

\subsection{$A$ Strategy}

We prove later in this section that each $A_n$ has value $0$ or $1$.
The following strategy $\mathcal{S}$ proves Left (Right) wins going first on $A_n$ ($-A_n$) when $n$ is odd and that the previous-player wins when $n$ is even.

As noted, within the tree $A_n$ there are subtrees isomorphic to $-A_m$ for $0 \le m < n$. The strategy $\strat{S}$ is also defined for playing  \pak on  $-A_n$ for all $n$.

\begin{Definition}
Fix $n \ge 0$. The strategy $\strat{S}$ is defined when playing \pak on some graph $A_n$ or its negative $-A_n$. Let $W$ be $A_n$ or $-A_n$.

The strategy $\mathcal{S}$ describes play on followers of $W$ reached with at least one player following the strategy \strat{S} on $W$ and winning.

Suppose $n$ is even. The strategy $\strat{S}$ is defined on $W$ for either player playing second.

Now suppose $n$ is odd.
If $W$ is $A_n$ then the strategy $\strat{S}$ is defined for Left playing first or second.
If $W$ is $-A_n$ then the strategy $\strat{S}$ is defined for Right playing first or second.
\begin{description}

\item [When $n$ is even] The strategy \strat{S} is defined  for the second player to move.

Suppose the first player plays on the edge $e$.
The second player responds by removing an edge of their color closer to the root node at distance one from $e$, and incident with the largest subgame of the form $A_\ell$ or $-A_\ell$, as appropriate for the player.
 
We are now possibly playing on a sum of $A_i$s and $-A_i$ for various $i < n$ along with what is left of $W$.

In the sequel we usually respond locally to the first player. There are three cases:
\begin{enumerate}
\item If the first player now plays in some $A_i$ or $-A_i$ then we respond locally following \strat S if we can. 

\item If the first player now plays in the remnants of a previously played-in tree, then we play as if we were following \strat S and responding to that move had it been the first move played in that summand.

\item If the first player now plays adjacent to the root of some summand then we play adjacent to the root of some untouched summand that has our color adjacent to the root.
\end{enumerate}

\item[When $n$ is odd]
The strategy \strat{S} is defined for Left playing first or second on $A_n$ and Right playing first or second on $-A_n$.

We first consider the cases where the strategy \strat{S} is defined for the first player to move.
The strategy \strat{S} for Left playing first is for 
Left to delete an edge from the root to $-A_{n-1}$.
The strategy \strat{S} for Right playing first is for 
Right to delete an edge from the root to $A_{n-1}$.
For either player playing second, the strategy $\strat{S}$ is as defined above in the case where $n$ is even.
\end{description}

\end{Definition}

In what follows, we use $W$ to denote some $A_n$ or $-A_n$ when the sign and index are not relevant.

The main theorem of this section, Theorem \ref{thm:Tn=pmn}, is a proof that this is a winning strategy for the conditions described. Also, Theorem \ref{thm:Tn=pmn} proves that the strategy proves that the strategy $\strat{S}$ is well-defined.

Figure \ref{fig:AnExtended} shows a subtree  a Right move, and a Left response. 

In Figure 5 the Edge $xy$ is red with $x$ the root of a subtree isomorphic to $-A_k$.

Figure 5 shows Right deleting the edge $e = xy$. This Figure captures a general situation with the exceptions that it not not represent the very simple game $A_1$ (which has no red edges); nor does it represent Right going first on some $-A_n$ and playing adjacent to the root.

To prove some Lemmas we want to be able to describe precisely the position of edges in some $W = A_n$ or $-A_n$.

We now introduce terminology to reference parts of a tree in respect to the strategy $\mathcal{S}$ and seen in Figure \ref{fig:AnExtended}. 

Let the distance from the root to $x$ be $d$ ($0 \le d \le n-1$).

The vertex $x$ is the root of a subtree isomorphic to some $A_k$ or some $-A_k$ (as shown in Figure 5).

The vertex $y$ is the root of a subtree isomorphic to some $-A_j$ or some $A_j$ (as shown in Figure 5).

Of particular interest to us is the value of the forest disconnected from the root of $W$ upon a move by either player and the opponent's response following the strategy $\mathcal{S}$. We denote this by $Q(W,e)$, where $e$ is the first edge played in $W$ and $e$ is not adjacent to the root.

The constraints on the variables that describe the position of the edge $e$ in some $W$ are:
\begin{itemize}
	\item $0 \le j \le k -1 < \ell - 1 < n$
	\item $ d \le n - \ell + 1$
\end{itemize}

\begin{figure}[h]
	\centering
	
	\begin{tikzpicture}[scale=1, every node/.style={draw, circle, minimum size=0.2cm, inner sep=1pt}]
		\node[label=left:{$r_e$}] (r) at (0,0) {};		
		\node (A1) at (-2,-2) {};
		\node[below=0.01cm of A1, draw=none] {$-A_1$};
		\node (A2) at (-1,-2) {};
		\node[below=0.01cm of A2, draw=none] {$-A_1$};
		\node[draw=none] at (-0.5, -2) {$\cdots$};
		\node (x) at (0,-2) {};
		\node[right=0.1cm of x, draw=none] {$x$};
		\node (An-1) at (3,-2) {};
		\node[right=0.01cm of An-1, draw=none] {$-A_{\ell-1}$};
		\node[draw=none] at (1.5, -2) {$\cdots$};
		
		\node (A1-x) at (-2,-4) {};
		\node[below=0.01cm of A1-x, draw=none] {$A_1$};
		\node (A1-x2) at (-1.25,-4) {};
		\node[below=0.01cm of A1-x2, draw=none] {$A_1$};
		\node (y) at (0,-4) {};
		\node[above left=0.1cm of y, draw=none] {$y$};
		\node (A_k-1) at (1,-4) {};
		\node[below=-0.15cm of A_k-1, draw=none] {$A_{k-1}$};
		\node[draw=none] at (-0.5,-4) {$\cdots$};
		\node[draw=none] at (0.5,-4) {$\cdots$};    
		
		\node (A1-y1) at (-1,-6) {};
		\node[below left=0.01cm of A1-y1, draw=none] {$-A_1$};
		\node (A1-y2) at (-0.15,-6) {};
		\node[below left=0.01cm of A1-y2, draw=none] {$-A_1$};
		\node (A_j-1) at (1,-6) {};
		\node[below right=-0.12cm of A_j-1, draw=none] {$-A_{j-1}$};
		\node[draw=none] at (0.5,-6) {$\cdots$};    
		
		\node (bA1-y1) at (2,-4) {};
		\node[below=0.01cm of bA1-y1, draw=none] {$A_1$};
		\node (bA1-y2) at (2.5,-4) {};
		\node[below=0.01cm of bA1-y2, draw=none] {$A_1$};
		\node (bA_j-1) at (3.5,-4) {};
		\node[below=-0.15cm of bA_j-1, draw=none] {$A_{\ell-2}$};
		\node[draw=none] at (3,-4) {$\cdots$};    
		
		\draw[red, dashed, line width=1pt] ([yshift=0.5cm]r.center) -- (r);
		
		\draw[blue] (r) -- (A1);
		\draw[blue] (r) -- (A2);
		\draw[blue] (r) -- (x);
		\draw[blue, line width=2pt] (r) -- (An-1);
		
		\draw[red] (x) -- (A1-x);
		\draw[red] (x) -- (A1-x2);
		\draw[red, line width=2pt] (x) -- (y);
		\draw[red] (x) -- (A_k-1);
		
		\draw[blue] (y) -- (A1-y1);
		\draw[blue] (y) -- (A1-y2);
		\draw[blue] (y) -- (A_j-1);
		
		\draw[red] (An-1) -- (bA1-y1);
		\draw[red] (An-1) -- (bA1-y2);
		\draw[red] (An-1) -- (bA_j-1);    
		
	\end{tikzpicture}
	\begin{tikzpicture}[scale=1, every node/.style={draw, circle, minimum size=0.2cm, inner sep=1pt}]
	\node[label=left:{$r_e$}] (r) at (0,0) {};
	\node (A1) at (-2,-2) {};
	\node[below=0.01cm of A1, draw=none] {$A_1$};
	\node (A2) at (-1,-2) {};
	\node[below=0.01cm of A2, draw=none] {$A_1$};
	\node[draw=none] at (-0.5, -2) {$\cdots$};
	\node (x) at (0,-2) {};
	\node[right=0.1cm of x, draw=none] {$x$};
	\node (An-1) at (3,-2) {};
	\node[right=0.01cm of An-1, draw=none] {$A_{\ell-1}$};
	\node[draw=none] at (1.5, -2) {$\cdots$};
	
	\node (A1-x) at (-2,-4) {};
	\node[below=0.01cm of A1-x, draw=none] {$-A_1$};
	\node (A1-x2) at (-1.25,-4) {};
	\node[below=0.01cm of A1-x2, draw=none] {$-A_1$};
	\node (y) at (0,-4) {};
	\node[above left=0.1cm of y, draw=none] {$y$};
	\node (A_k-1) at (1,-4) {};
	\node[below=-0.15cm of A_k-1, draw=none] {$-A_{k-1}$};
	\node[draw=none] at (-0.5,-4) {$\cdots$};
	\node[draw=none] at (0.5,-4) {$\cdots$};    
	
	\node (A1-y1) at (-1,-6) {};
	\node[below left=0.01cm of A1-y1, draw=none] {$A_1$};
	\node (A1-y2) at (-0.25,-6) {};
	\node[below left=0.01cm of A1-y2, draw=none] {$A_1$};
	\node (A_j-1) at (1,-6) {};
	\node[below right=-0.12cm of A_j-1, draw=none] {$A_{j-1}$};
	\node[draw=none] at (0.5,-6) {$\cdots$};
	
	\node (bA1-y1) at (2,-4) {};
	\node[below=0.01cm of bA1-y1, draw=none] {$-A_1$};
	\node (bA1-y2) at (3,-4) {};
	\node[below=0.01cm of bA1-y2, draw=none] {$-A_1$};
	\node (bA_j-1) at (4.5,-4) {};
	\node[below=-0.15cm of bA_j-1, draw=none] {$-A_{\ell-2}$};
	\node[draw=none] at (3.75,-4) {$\cdots$};    
	
	\draw[blue, dashed, line width=1pt] ([yshift=0.5cm]r.center) -- (r);
			
	\draw[red] (r) -- (A1);
	\draw[red] (r) -- (A2);
	\draw[red] (r) -- (x);
	\draw[red, line width=2pt] (r) -- (An-1);
	
	\draw[blue] (x) -- (A1-x);
	\draw[blue] (x) -- (A1-x2);
	\draw[blue, line width=2pt] (x) -- (y);
	\draw[blue] (x) -- (A_k-1);
	
	\draw[red] (y) -- (A1-y1);
	\draw[red] (y) -- (A1-y2);
	\draw[red] (y) -- (A_j-1);
	
	\draw[blue] (An-1) -- (bA1-y1);
	\draw[blue] (An-1) -- (bA1-y2);
	\draw[blue] (An-1) -- (bA_j-1);    
	
\end{tikzpicture}
	
	\caption{Left's response to Right playing $e = xy$ under the strategy $\mathcal{S}$ in the game $A_n$, played on some graph $W$. The figure shows the subgraph of $W$ rooted at $r_e$.
	The dashed line indicates the connection of this subtree up towards the root of $W$, namely $r$.
	}
	\label{fig:AnExtended}
\end{figure}
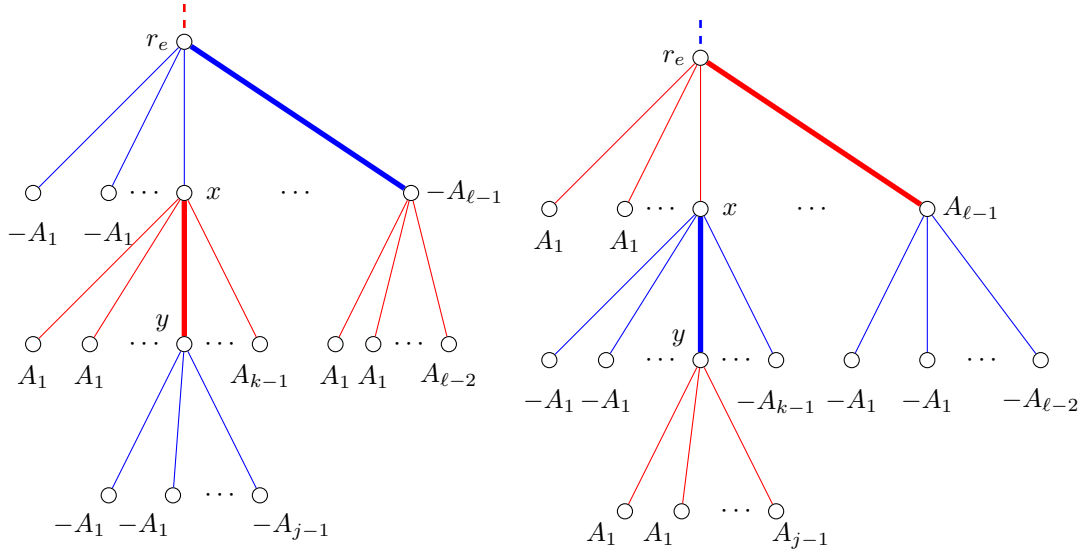

\begin{figure}[h]
	\centering
	
	\begin{tikzpicture}[scale=1, every node/.style={draw, circle, minimum size=0.2cm, inner sep=1pt}]
		\node[label=left:{$r_e$}] (r) at (0,0) {};
		\node (A1) at (-2,-2) {};
		\node[below=0.01cm of A1, draw=none] {$-A_1$};
		\node (A2) at (-1,-2) {};
		\node[below=0.01cm of A2, draw=none] {$-A_1$};
		\node[draw=none] at (-0.5, -2) {$\cdots$};

		\node (An-1) at (3,-2) {};
		\node[right=0.01cm of An-1, draw=none] {$(-A_{\ell-1})$};
		\node[draw=none] at (1.5, -2) {$\cdots$};
		
		\node (A1-x) at (-2,-4) {};
		\node[below=0.01cm of A1-x, draw=none] {$A_1$};
		\node (A1-x2) at (-1.25,-4) {};
		\node[below=0.01cm of A1-x2, draw=none] {$A_1$};

		\node (A_k-1) at (1,-4) {};
		\node[below=-0.15cm of A_k-1, draw=none] {$A_{k-1}$};
		\node[draw=none] at (-0.5,-4) {$\cdots$};
		\node[draw=none] at (0.5,-4) {$\cdots$};    
		
		\node (A1-y1) at (-1,-6) {};
		\node[below left=0.01cm of A1-y1, draw=none] {$(-A_1)$};
		\node (A1-y2) at (-0.25,-6) {};
		\node[below=-0.2cm of A1-y2, draw=none] {$(-A_1)$};
		\node (A_j-1) at (1,-6) {};
		\node[below right=-0.12cm of A_j-1, draw=none] {$(-A_{j-1})$};
		\node[draw=none] at (0.5,-6) {$\cdots$};    
		
		\node (bA1-y1) at (2,-4) {};
		\node[below=0.01cm of bA1-y1, draw=none] {$A_1$};
		\node (bA1-y2) at (2.5,-4) {};
		\node[below=0.01cm of bA1-y2, draw=none] {$A_1$};
		\node (bA_j-1) at (3.5,-4) {};
		\node[below=-0.15cm of bA_j-1, draw=none] {$A_{\ell-2}$};
		\node[draw=none] at (3,-4) {$\cdots$};    
		
		\draw[red, dashed, line width=1pt] ([yshift=0.5cm]r.center) -- (r);
		
		\draw[blue] (r) -- (A1);
		\draw[blue] (r) -- (A2);

		\draw[blue, line width=2pt] (r) -- (An-1);
		
		\draw[red] (An-1) -- (bA1-y1);
		\draw[red] (An-1) -- (bA1-y2);
		\draw[red] (An-1) -- (bA_j-1);    
		
	\end{tikzpicture}
	\begin{tikzpicture}[scale=1, every node/.style={draw, circle, minimum size=0.2cm, inner sep=1pt}]		
		\node (A1) at (-2,-2) {};
		\node[below=0.01cm of A1, draw=none] {$-A_1$};
		\node (A2) at (-1,-2) {};
		\node[below=0.01cm of A2, draw=none] {$-A_1$};
		\node[draw=none] at (-0.5, -2) {$\cdots$};
		
		\node (A1-x) at (-2,-4) {};
		\node[below=0.01cm of A1-x, draw=none] {$A_1$};
		\node (A1-x2) at (-1.25,-4) {};
		\node[below=0.01cm of A1-x2, draw=none] {$A_1$};

		\node (A_k-1) at (1,-4) {};
		\node[below=-0.15cm of A_k-1, draw=none] {$A_{k-1}$};
		\node[draw=none] at (-0.5,-4) {$\cdots$};
		\node[draw=none] at (0.5,-4) {$\cdots$};    
		
		\node (A1-y1) at (-1,-6) {};
		\node[below left=0.01cm of A1-y1, draw=none] {$(-A_1)$};
		\node (A1-y2) at (-0.25,-6) {};
		\node[below=-0.2cm of A1-y2, draw=none] {$(-A_1)$};
		\node (A_j-1) at (1,-6) {};
		\node[below right=-0.12cm of A_j-1, draw=none] {$(-A_{j-1})$};
		\node[draw=none] at (0.5,-6) {$\cdots$};    
		
		\node (bA1-y1) at (2,-4) {};
		\node[below=0.01cm of bA1-y1, draw=none] {$A_1$};
		\node (bA1-y2) at (2.5,-4) {};
		\node[below=0.01cm of bA1-y2, draw=none] {$A_1$};
		\node (bA_j-1) at (3.5,-4) {};
		\node[below=-0.15cm of bA_j-1, draw=none] {$A_{\ell-2}$};
		\node[draw=none] at (3,-4) {$\cdots$};

	\end{tikzpicture}
	
	\caption{The game position after Right plays, then after Left responds (Or we just use the first image, then say that all visible edges here are removed after Left's subsequent response)}
	\label{fig:AnMoves}
\end{figure}
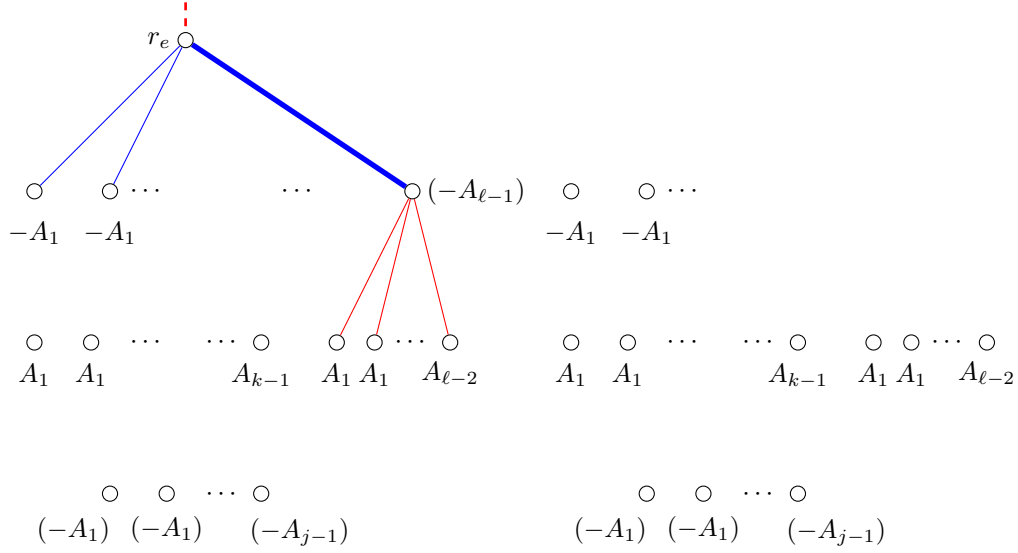

\subsection{The value of $A_n$ is $0$ or $1$ and has the same parity as $n$.}\label{sec:proof}
\leavevmode

\smallskip

When playing on some $W$ using the strategy $\strat{S}$, particular sums of games appear and they simplify greatly.
Next we present some algebra that is helpful when we consider the value of each of these positions.

\begin{Lemma}\label{lem:mod2sum}
Let $m_i = i\pmod 2$ for $n\geq 0$. Then $2\sum_{i=1}^{k} m_i = k + m_k$.
\end{Lemma}
\begin{proof}
If $k$ is even, then $2\sum_{i=1}^{k} m_i = 2\sum_{i=1}^{k/2}m_{2i-1} =2\sum_{i=1}^{k/2}1 = k$.

If $k$ is odd, then $2\sum_{i=1}^{k} m_i = 2\sum_{i=1}^{\frac{k+1}{2}}m_{2i-1} = 2\sum_{i=1}^{\frac{k+1}{2}}1 = k+1 = k+m_k$.
\end{proof}

\begin{Theorem}\label{thm:AnValue}
The value of the \pak position $A_n$ for $n \ge 0$ is as follows:

\[
A_n = 
\begin{cases}
0 &\text{if } n \equiv 0 \pmod{2}\\
1 &\text{if } n \equiv 1 \pmod{2}
\end{cases}.
\]
\end{Theorem}

\begin{proof}
Let $p(n)$ be $0$ if $n$ is even and $1$ if $n$ is odd.

We proceed using induction:
for $n\geq 0$, 
\begin{itemize}
\item the \pak position $A_n$ has value $p(n)$
\item $A_n - p(n)$ is a Previous-player win following the strategy $S$.
\end{itemize}

The base cases of $n = 0,1,2,3$ are all true by inspection. Assume the claim is true for all $A_n$ with $0 < n < m$ for some integer $m$. Note that when $n<m$:

\begin{enumerate}
	\item $ 2\sum_{i=0}^n A_i = 2\sum_{i=1}^n A_i = n+A_n$ (by Lemma \ref{lem:mod2sum})
	\item $A_n+A_{n-1} = 1$, for $n \ge 1$ 
\end{enumerate}

Let $n=m$.

We consider play with each player moving first.

Again, $Q(W,e)$ is the disconnected sum after the move in the edge $e$ and the response following the strategy $\strat{S}$.

Using our assumptions, $Q(W,e)$ simplifies. When $e$ is red it simplifies as follows:
\begin{align*}
Q(W,e) &= (-2\sum_{i=1}^{\ell-1}A_i + A_k + A_{\ell-1}) + (2\sum_{i=1}^{k-1}A_i - A_j) + (2\sum_{i=1}^{\ell-2}A_i) + (-2\sum_{i=1}^{j-1}A_i) \\
&=(-(\ell-1)-A_{\ell-1}+A_k+A_{\ell-1})  + (k-1+A_{k-1} - A_j)  \\
& + (\ell-2+A_{\ell-2})  + (-(j-1)-A_{j-1}) \\
&= -\ell+1+(A_{k-1}+A_k)+k-1-(A_{j-1}+A_j)+\ell-2-j+1+A_{\ell-2} \\
&= -\ell+1+(1)+k-1-(1)+\ell-2-j+1 + A_{\ell-2}\\
&= (k-j-1) + A_{\ell-2}.
\end{align*}

A symmetric argument shows that when $e$ is blue,
\[Q(W,e) = -(k-j-1 +A_{\ell-2}) = -(k-j-1) - A_{\ell-2}.\]

The value of $Q(W,e)$ is, by our inductive assumption, a number and an integer. Furthermore, in this construction $k-j -1 \ge 0 $. In summary,

if $e$ is red, $ Q(W,e) = k-j-1 +A_{\ell-2} \ge A_{\ell-2}$;

if $e$ is blue, $ Q(W,e) = 	-(k-j-1) -A_{\ell-2}  \le -A_{\ell-2}$.

Let $r*$ denote the root of $A_n$.
For any edge $e$, let $r_e$ denote the vertex one closer to $r*$ above $x$ in the tree, as shown in Figure \ref{fig:AnExtended}. No matter the distance from $r*$ to $r_e$, after playing $e$ and $\mathcal{S}(e)$, the vertex $r_e$ is removed from the resulting subtree (possibly empty) of $W$ still connected to $r*$; 
this subtree of $W$ is the same as would be reached by Left playing any blue edge adjacent to $r_e$ in $W$.

We first consider the case where Right plays first. Later we consider when Left plays first. 

\begin{description}
\item[Right moving first]

With Left following the strategy $\strat{S}$, 
after some even number of moves Left plays near the root.
For $n \ge 2$ in $A_n$, there are two blue edges adjacent to the root that are incident to pendant red edges. This assures that there will always a move for Right and subsequently a response for Left until Left plays incident with the root.
Let $t$ be the total number of Right moves (including the Right move playing distance $1$ from the root); this is also the number of Left responses. That is, Right's most recent move is in the edge $e_{t}$.

By Lemma \ref{thm:PAKstrong}, at the end of a play the moves played (and not the order in which moves were played) determine the position reached.

We now argue that after Left's response to Right's move $e_t$ at distance $1$ from the root, Left wins the remaining sum.
If Right started with this move, that is, $t = 1$, then the remaining sum is exactly $Q(W,e_1)$ and our argument is complete by induction.

Otherwise, we consider each of the $t$ pairs of moves (each Right then Left) in a breadth-first manner.
After each pair of moves starting with Right deleting some $e_i$ (where $1 \le i \le t$), which is red, and with Left following $\strat{S}$ and deleting a blue edge, the value of the components not connected to the root changes by $Q(A_{\pi(i)},e_i)$ where $A_{\pi(i)}$ is isomorphic to the subtree with root $r_{e_i}$. Note that $\pi(i)$ has the same parity as $n$ (as $r^*$ and $r_{e_i}$ both have blue child edges) and $\pi(i) < n$.

Recall that we've already played the game to the point where Right played distance $1$ from the root and the eventual connected component is empty. That is, following $\strat{S}$ we have shown that Left eventually plays and reaches a position of the form \[ \sum_{i} Q(A_i,e_i).\]

Recall 
\[ \sum_{i} A_{\pi(i)} \ge  A_{p(n)} \qquad \mbox{ and } \qquad k_i-j_i-1 \ge 0 .\] 

By induction \[ \sum_{1 \le i \le t} Q(A_{\pi(i)},e_i) = \sum_{i} A_{\pi(i)} + \sum_{i} (k_i-j_i-1) \ge A_{p(n)} + 0 \ge p(n).\]

We have shown that, when $n\ge 4$, Left playing second following $\strat{S}$ in $A_n$ moves to a position that is at least $p(n)$. That is, Right loses going first in $A_n - p(n)$.

\item[Left moving first]

Now we consider Left moving first on $A_n$. Up to symmetry, the case where $n$ is even is essentially the same as the case where Right goes first. That is, Left loses going first in $A_n - p(n)$.

Assume $n$ is odd. 
It remains to show that Left loses moving first on $A_n -1$.

We claim that Left loses moving first on $A_n - 1$ if Right follows \strat{S} playing on $A_n$.
 
After some even number of moves where Right follows the strategy $\strat{S}$, Left will play in an edge incident to the root.
Note that the tree of $A_n$ is a tree composed of a root vertex adjacent (via blue edges) to vertices that are roots of subtrees isomorphic to $-A_{\nu}$ for some odd $\nu$. 
Also, the responses dictated by $\strat{S}$ are local to these subtrees. Following $\strat{S}$ on $A_n$ is the the same as following $\strat{S}$ on each of these subtrees.
Thus each $-A_{\nu}$ for odd $\nu$ is winning for Right as the second player following $\strat{S}$ and thus Right playing second on $A_n$ if Left is not allowed to play adjacent to the root.
 
Let $t$ be the total number of Left moves including the Left move playing distance $0$ from the root if that happens. The number of Right responses in $A_n -1$ is also $t$.

That is, Left's last move is in the edge $e_{t}$ but Right will, when unable to follow $\strat{S}$, have one additional move in $-1$. Note we are following $\strat{S}$ playing on $A_n$ but we are extending that to have a strategy on the sum $A_n - 1$.

By Lemma \ref{thm:PAKstrong}, at the end of a play the order in which moves were played does not change which position is reached. This extends to $A_n -1$ because taking the disjoint union of a \pak position with a red edge is also a \pak\ position.

If Left started with the move distance $0$ from the root, that is, $t = 1$, then the remaining sum is exactly $-1$ and Right wins playing this to $0$.

We now argue that after Right's response to Left's move $e_t$ at distance at least $2$ from the root, Right wins the remaining sum.

Otherwise, we consider each of the $t$ pairs of moves (each  Right then Left) in a breadth-first manner. 

After each pair of moves, starting with Left deleting some $e_i$ (where $1 \le i \le t$), which is blue, and with Right following $\strat{S}$ and deleting a red edge, the value of the disconnected components changes by $Q(A_{\pi(i)},e_i)$ where $A_{\pi(i)}$ is isomorphic to the subtree with root $r_{e_i}$. Note that $\pi(i)$ has the same parity as $n$ (as $r^*$ and $r_{e_i}$ both have blue child edges) and $\pi(i) < n$.

Recall that we've already played the game to the point where Left played adjacent (distance $0$) from the root and there is no component connected to the root because there is no root. 

That is, by following $\strat{S}$  in $A_n$, Right eventually reaches a position of the form \[ \sum_{i} Q(A_i,e_i).\]

Recall 
\[ \sum_{i} A_{\pi(i)} \ge  A_{p(n)} \qquad \mbox{ and } \qquad k_i-j_i-1 \ge 0 .\] 

By induction \[ \sum_{1 \le i \le t} Q(A_{\pi(i)},e_i) = \sum_{i} A_{\pi(i)} + \sum_{i} (k_i-j_i-1) \ge A_{p(n)} + 0 \ge p(n).\]

We have shown that, when $n\ge 4$, Left playing second following $\strat{S}$ in $A_n$ moves to a position that is at least $p(n)$. That is, Left loses going first in $A_n - p(n)$.

\end{description}

Whereas Left and Right both lose going first in $A_n - p(n)$, then $A_n = p(n)$.

\end{proof}

The strategy $\strat{S}$ is defined in exactly the cases where there exists a winning strategy for the player.

\begin{Corollary}
	The strategy $\mathcal{S}$ is winning.
\end{Corollary}

\section{High Temperature Positions}\label{sec:High}

This section contains the main result of this paper.

\subsection{Hot Construction}

Theorem \ref{thm:AnValue} is sufficient to show that positions of the form $C_n$ (introduced in Section \ref{sec:A}) have high temperature, but we construct positions, called $T_n$, that have with twice the temperature with the same highest degree. These positions are symmetric (isomorphic to itself under the transformation of swapping all colors blue to red and vice-versa); unlike $C_n$, $T_n$ is its own negative.

\begin{Construction}\label{T}
	For any integer $n>1$ consider the \pak position, denoted $T_n$, constructed in the following way:
	\begin{itemize}
		\item Let $u$ be the root of a copy of $A_n$.
		\item Let $w$ be the root of a copy of $-A_n$.
		\item Add a vertex $v$ adjacent to $u$ via a red edge and adjacent to $w$ via a blue edge.
	\end{itemize}
\end{Construction}

See Figure \ref{fig:T3} for the construction of $T_3$. We will call the edge $uv$ Right's central edge and the edge $vw$ Left's central edge.

\begin{Theorem}\label{thm:Tn=pmn}
If $n \ge 0$, then the position $T_n$ has value $\pm n$.
\end{Theorem}
\begin{proof}
Note that $T_0 = \ast$.
For $n > 0$, we play the difference with $H= \pm n$ and show that the second player wins.

The position $T_n$ is symmetric about the central vertex $v$ and thus $T_n = -T_n$. Also, $H = -H$.
It suffices to argue only about Left moving first in this sum by symmetry.

That is, we play $T_n -\left( \pm n\right) = T_n + \pm n$. Hence we refer to playing a sum.

Left's deletion of her central edge results in a position, which we denote by $T_n^{L_0}$, that is a copy of $A_n$ along with two copies of each of $A_1, A_2, \ldots , A_{n-2}$.

The value of $T_n^{L_0}$ is $ \sum_{m=1}^{n-1} 2\left(A_m \right) + A_n = n-1 + A_{n-1} + A_n = n$,  which follows from Lemma \ref{lem:mod2sum} and Theorem \ref{thm:AnValue}.
Similarly, if Right removes his central edge, then what remains has value equal to $-n$.

If Left starts and plays $T_n \xrightarrow{L}T_n^{L_0}$ in the sum, then Right can respond in the $H$ component of the sum to $T_n^{L_0} -n = 0$ and win.
Similarly, first moves in $H$ have winning responses in $T_n$.
It remains to show that other first moves for Left, which are in $T_n$, are losing in the sum.
We show that if Left makes such a move that Right wins responding in the switch to $-n$.

Suppose Left moves first in the sum in some non-central blue edge and Right responds in the switch to $-n$. If Left's second move is in the central blue edge we reach a position that is a Left option of $T_n^{L_0} -n$. This is a Left option of $T_n^{L_0} - n$, a game with value $0$, so Right wins going first from this position.
If instead Left's second move is in a non-central blue edge, Right must respond.
To complete the proof we show that Right wins for any subsequent Left response. We need to consider three cases. 

The cases are described by moves played from $T_n + \pm n$. Cases are reached in multiple ways, but  permutations of  Right's moves and Left's moves reach the same position by Lemma \ref{thm:PAKstrong}. For justification, note that the graph $K_{1,n}$ in \snort, a strong placement game, has value $\pm n$; hence we can consider the value $\pm n$ as a strong placement game. Furthermore, a sum of strong placement games is a strong placement game \cite{FaridHN2019}.

\begin{enumerate}
\item 

Left plays some non-central blue edge in $T_n$, then Right plays in the switch $H$, and then Left plays her central edge.
This case is equivalent to when Left plays her central edge in $T_n$, then Right plays in the switch, and then Left plays some non-central edge. We proved above that Right wins going first from this position.

\item Left plays some non-central blue edge, then Right plays in the switch $H$, and then Left plays some blue edge non-incident with the central red (which is thus also non-central).

Right's winning response is to play his central edge. Heuristically, it is easy to see that Right wins: Left has played twice in components that are numbers in a sum equal to $-2n$ and thus loses. However, we need to be careful because Left may play twice in a row in a summand and the standard theory isn't particularly helpful here. We conclude the argument that Right wins below because it is the same argument shows that Right wins in the third case.

\item Left plays a blue edge, Right plays in the switch $H$, and then Left plays the blue edge incident with central red.

Right's winning response is to play mirroring Left's play that was incident to the central red edge. This move exists because the construction of $A_n$ has two isomorphic children. After this Right move it is Left's turn to move from a Left option of sum equal to $-n$.
\end{enumerate}

Each summands in these sums is an $A_n$ or a $-A_n$ (or a negative integer). 
Whenever Left makes a number of moves in a sum of $W$s equal to a negative number (plus possibly a negative integer in addition), Right wins. Component by component, Right can respond following the strategy $\strat{S}$ for any component in which Left has played. The difference from the original sum (which was negative) is some $Q(A,e)$ in Right's favour. Eventually, Left runs out of moves and reaches a sum at most equal to the negative sum we started with.

\end{proof}

\begin{Corollary}\label{cor:highTemp}
The ruleset \pak has unbounded temperature.
\end{Corollary}

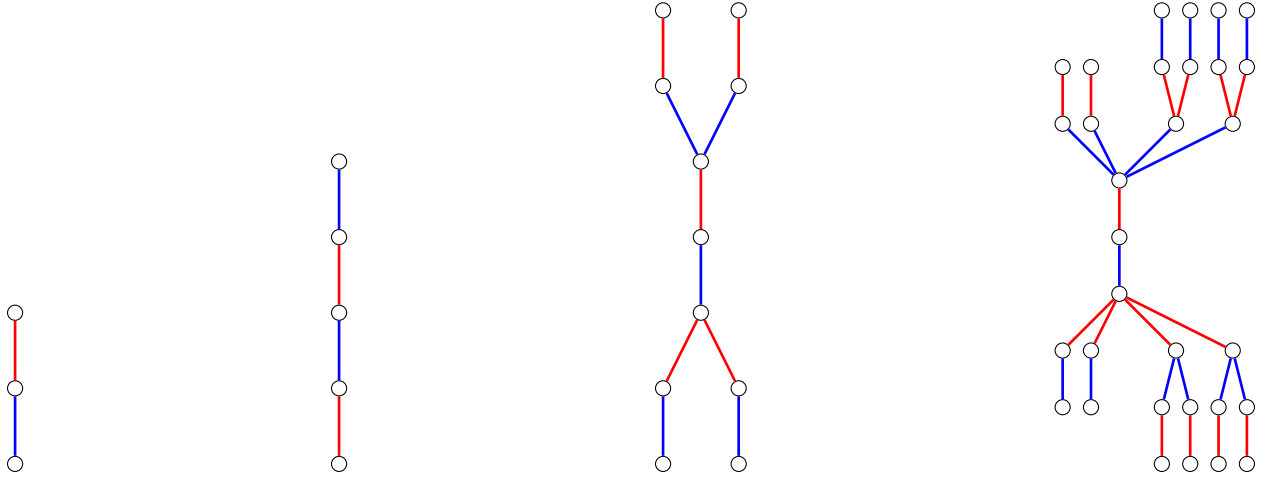
\begin{figure}[h]

\begin{tikzpicture}[scale=1, every node/.style={draw, circle, minimum size=0.2cm, inner sep=1pt, thin}, every path/.style={line width=1pt}]
	\node (c0) at (0,0) {};
	\node (r0) at (0,1) {};
	\node (b0) at (0,-1) {};
	\draw[red] (c0) -- (r0);
	\draw[blue] (c0) -- (b0);
\end{tikzpicture} \hfill
\begin{tikzpicture}[scale=1, every node/.style={draw, circle, minimum size=0.2cm, inner sep=1pt, thin}, every path/.style={line width=1pt}]
	\node (c1) at (0,0) {};
	\node (r1) at (0,1) {};
	\node (b1) at (0,-1) {};
	\node (b11) at (0,2) {};
	\node (r11) at (0,-2) {};
	
	\draw[red] (c1) -- (r1);
	\draw[blue] (c1) -- (b1);
	\draw[blue] (r1) -- (b11);
	\draw[red] (b1) -- (r11);
\end{tikzpicture} \hfill
\begin{tikzpicture}[scale=1, every node/.style={draw, circle, minimum size=0.2cm, inner sep=1pt, thin}, every path/.style={line width=1pt}]
	\node (c2) at (0,0) {};
	\node (r2) at (0,1) {};
	\node (b2) at (0,-1) {};
	\node (b21) at (-.5,2) {};
	\node (b22) at (.5,2) {};
	\node (r21) at (-.5,-2) {};
	\node (r22) at (.5,-2) {};
	\node (r210) at (-.5,3) {};
	\node (r220) at (.5,3) {};
	\node (b210) at (-.5,-3) {};
	\node (b220) at (.5,-3) {};
	
	\draw[red] (c2) -- (r2);
	\draw[blue] (r2) -- (b21);
	\draw[blue] (r2) -- (b22);
	\draw[red] (b21) -- (r210);
	\draw[red] (b22) -- (r220);

	\draw[blue] (c2) -- (b2);
	\draw[red] (b2) -- (r21);
	\draw[red] (b2) -- (r22);
	\draw[blue] (r21) -- (b210);
	\draw[blue] (r22) -- (b220);	
\end{tikzpicture} \hfill
\begin{tikzpicture}[scale=.75, every node/.style={draw, circle, minimum size=0.2cm, inner sep=1pt, thin}, every path/.style={line width=1pt}]
	\node (c2) at (0,0) {};
	\node (r2) at (0,1) {};
	\node (b2) at (0,-1) {};
	\node (b21) at (-1,2) {};
	\node (b22) at (-.5,2) {};
	\node (r21) at (-1,-2) {};
	\node (r22) at (-.5,-2) {};
	\node (r210) at (-1,3) {};
	\node (r220) at (-.5,3) {};
	\node (b210) at (-1,-3) {};
	\node (b220) at (-.5,-3) {};

	\node (b31) at (2,2) {};
	\node (b32) at (1,2) {};
	\node (r31) at (2,-2) {};
	\node (r32) at (1,-2) {};
	\node (r310) at (2.25,3) {};
	\node (r311) at (1.75,3) {};	
	\node (r320) at (1.25,3) {};
	\node (r321) at (.75,3) {};	
	\node (b310) at (2.25,-3) {};
	\node (b311) at (1.75,-3) {};	
	\node (b320) at (1.25,-3) {};
	\node (b321) at (.75,-3) {};	
	\node (b400) at (.75,4) {};
	\node (b401) at (1.25,4) {};
	\node (b402) at (1.75,4) {};
	\node (b403) at (2.25,4) {};
	\node (r400) at (.75,-4) {};
	\node (r401) at (1.25,-4) {};
	\node (r402) at (1.75,-4) {};
	\node (r403) at (2.25,-4) {};		
	
	\draw[red] (c2) -- (r2);
	\draw[blue] (r2) -- (b21);
	\draw[blue] (r2) -- (b22);
	\draw[blue] (r2) -- (b31);
	\draw[blue] (r2) -- (b32);	
	\draw[red] (b21) -- (r210);
	\draw[red] (b22) -- (r220);
	\draw[red] (b31) -- (r310);
	\draw[red] (b31) -- (r311);
	\draw[red] (b32) -- (r320);
	\draw[red] (b32) -- (r321);

	\draw[blue] (c2) -- (b2);
	\draw[red] (b2) -- (r21);
	\draw[red] (b2) -- (r22);
	\draw[red] (b2) -- (r31);
	\draw[red] (b2) -- (r32);	
	\draw[blue] (r21) -- (b210);
	\draw[blue] (r31) -- (b310);
	\draw[blue] (r31) -- (b311);	
	\draw[blue] (r22) -- (b220);	
	\draw[blue] (r32) -- (b320);
	\draw[blue] (r32) -- (b321);	
	
	\draw[blue] (r310) -- (b403);
	\draw[blue] (r311) -- (b402);	
	\draw[blue] (r320) -- (b401);	
	\draw[blue] (r321) -- (b400);	
	\draw[red] (b310) -- (r403);
	\draw[red] (b311) -- (r402);	
	\draw[red] (b320) -- (r401);	
	\draw[red] (b321) -- (r400);		
\end{tikzpicture}
\caption{$T_n$ with temperature $n$ for $n= 0,1,2,\text{and } 3$}
\label{fig:T3}
\end{figure}

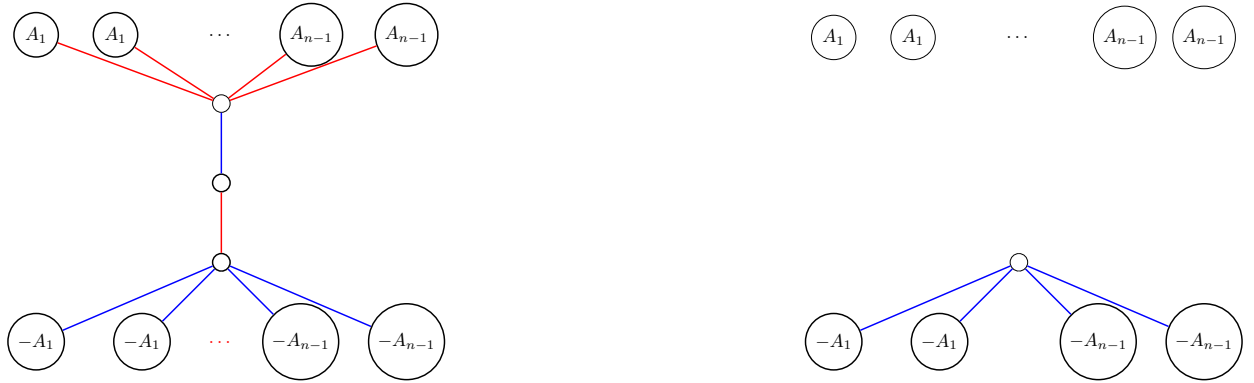
\begin{figure}
\centering

\scalebox{0.7}{
\begin{tikzpicture}
	[main node/.style={circle,draw=black,fill=white,text=black},
	main edge from parent/.style={draw,thick,draw=blue}]
	\node[main node] (root) {}
	child[main edge from parent] {node[main node] (mid) {}
	child[main edge from parent,red] {node[main node] (rootA) {}
	child[main edge from parent] {node[main node,xshift=-0.5cm] (a1) {$-A_1$}}
	child[main edge from parent] {node[main node] (a2) {$-A_1$}}
	child[draw=white] {node[] (x) {\ldots}}
	child[main edge from parent] {node[main node] (a3) {$-A_{n-1}$}}
	child[main edge from parent] {node[main node,xshift=0.5cm] (a4) {$-A_{n-1}$}}
}};

        \node[main node, yshift=1.3cm, xshift=-3.5cm, thick] (b1) at (root) {$A_1$};
        \draw[red,thick] (b1) -- (root);
        \node[main node, yshift=1.3cm, xshift=-2cm, thick] (b2) at (root) {$A_1$};
        \draw[red,thick] (b2) -- (root);
        \node[draw=white, yshift=1.3cm] (bx) at (root) {\ldots};
        \node[main node, yshift=1.3cm, xshift=1.7cm, thick] (b3) at (root) {$A_{n-1}$};
        \draw[red,thick] (b3) -- (root);
        \node[main node, yshift=1.3cm, xshift=3.5cm, thick] (b4) at (root) {$A_{n-1}$};
        \draw[red,thick] (b4) -- (root);
\end{tikzpicture}
}\hfill %
\scalebox{0.7}{
\begin{tikzpicture}
	[main node/.style={circle,draw=black,fill=white,text=black},
	main edge from parent/.style={draw,thick,draw=blue}]
	\node[] (root) {}
	child[white] {node[] (mid) {}
	child[white] {node[main node] (rootA) {}
	child[main edge from parent] {node[main node,xshift=-0.5cm] (a1) {$-A_1$}}
	child[main edge from parent] {node[main node] (a2) {$-A_1$}}
	child[draw=white] {node[] (x) {\ldots}}
	child[main edge from parent] {node[main node] (a3) {$-A_{n-1}$}}
	child[main edge from parent] {node[main node,xshift=0.5cm] (a4) {$-A_{n-1}$}}
}};

        \node[main node, yshift=1.25cm, xshift=-3.5cm] (b1) at (root) {$A_1$};
        \node[main node, yshift=1.25cm, xshift=-2cm] (b2) at (root) {$A_1$};
        \node[draw=white, yshift=1.25cm] (bx) at (root) {\ldots};
        \node[main node, yshift=1.25cm, xshift=2cm] (b3) at (root) {$A_{n-1}$};
        \node[main node, yshift=1.25cm, xshift=3.5cm] (b4) at (root) {$A_{n-1}$};
\end{tikzpicture}
}
\caption{$T_n$ and Left's option $T_n^L = A_n + 2\sum_{i=1}^{n-1}A_i$}
\end{figure}

\section{Conclusion}\label{sec:Conclusion}

The positions $T_n$ for $n \ge 3$ have temperature at least $3$, which exceeds the greatest known temperature of a \dom position.

The $A_n$ construction gives infinitely-many non-trivial and asymmetric zero positions.

\subsection{Future Work}

Our work shows that when no grid structure is required at all, then there is no bound to the temperature.
Hence, if the boiling point of \dom is finite, it is in some way because of the grid structure of a \dom board. We can restrict the \pak positions considered to those closer to a \dom position though and it is yet unknown whether those families would have a finite boiling point. Two families to consider are:
\begin{itemize}
\item \textbf{Degree restrictions}: Each vertex in a \pak position equivalent to a \dom position has degree at most $4$, with at most $2$ blue edges and at most $2$ red edges. The $T_n$ family has unbounded maximum degree, so a natural restriction would be to consider only those \pak positions with a bounded maximum degree and/or those where the number of edges of one color incident with each vertex is bounded.
\item \textbf{Grid structure}: We can also consider the family of \pak positions which can be embedded in a Cartesian grid, but relaxing the \dom condition of having vertical edges blue and horizontal edges red. Since we are restricting to graphs with maximum degree at most $4$, these positions are a sub-family of the previous family. Thus if the temperature is still found to be unbounded (or have a constant bound larger than $2$) when restricting the degree, this family would be the next obvious one to consider.
\end{itemize}

Although not entirely motivated by the question of the boiling point of \dom, it would also be interesting to consider whether the sparseness, tree-structure, or alternating coloring of the graphs $T_n$ could be affecting the temperature. In particular, can we still have unbounded temperature in dense graphs or in graphs of large girth? Are there other families of trees with alternating coloring and high degree where the temperature grows more quickly compared to the size of the board?

\end{document}